\documentclass[11pt]{article}

\usepackage{amsmath}
\usepackage{amsfonts}
\usepackage{amssymb}
\usepackage{amsthm}
\usepackage{enumitem}
\usepackage{mathtools}
\usepackage[letterpaper, margin=1in]{geometry}
\usepackage{setspace}
\usepackage{etoolbox}
\usepackage{color}

\allowdisplaybreaks

\newtheorem{theorem}{Theorem}[section]
\newtheorem{lemma}{Lemma}[section]
\newtheorem{proposition}{Proposition}[section]
\newtheorem{corollary}{Corollary}[section]

\theoremstyle{definition}
\newtheorem{definition}{Definition}[section]
\newtheorem{example}{Example}[section]
\newtheorem{remark}{Remark}[section]

\renewenvironment{proof}[1][Proof.]{\begin{trivlist}
		\item[\hskip \labelsep {\bfseries #1}]}{\end{trivlist}}

\newcommand{\gap}{\vspace{2ex}}
\newcommand{\smallgap}{\vspace{1ex}}

\newcommand{\conv}{\operatorname{conv}}

\newcommand{\core}[1]{\langle #1 \rangle}
\newcommand{\ext}{\operatorname{ext}}
\newcommand{\spn}{\operatorname{span}}

\newcommand{\epi}{\operatorname{epi}}
\newcommand{\ran}{\operatorname{ran}}
\newcommand{\R}{\mathcal{R}}
\newcommand{\Rn}{\mathcal{R}^n}
\newcommand{\Sn}{\mathcal{S}^n}
\newcommand{\Hn}{\mathcal{H}^n}
\newcommand{\V}{\mathcal{V}}
\newcommand{\W}{\mathcal{W}}
\newcommand{\G}{\mathcal{G}}
\newcommand{\lc}{\lambda(c)}
\newcommand{\lx}{\lambda(x)}
\newcommand{\ly}{\lambda(y)}
\newcommand{\lu}{\lambda(u)}
\newcommand{\barx}{\overline{x}}
\newcommand{\abs}[1]{\left\vert #1 \right\vert}
\newcommand{\norm}[1]{\left\Vert #1 \right\Vert}
\newcommand{\ip}[2]{\left< #1, #2 \right>}

\title{\bf Transfer  principles, Fenchel conjugate and  subdifferential formulas  in  \\
	Fan-Theobald-von Neumann systems
}
\author{Juyoung Jeong \\
	Applied Algebra and Optimization Research Center \\
	Sungkyuankwan University \\
	2066 Seobu-ro, Suwon 16419, Republic of Korea \\
	jjycjn@skku.edu\\[1ex]
	and\\[1ex]
	M. Seetharama Gowda\\
	Department of Mathematics and Statistics\\
	University of Maryland, Baltimore County\\
	Baltimore, Maryland 21250, USA\\
	gowda@umbc.edu}

\date{\today}

\begin{document}
	
\maketitle\vspace{1cm}

\begin{abstract}
    A Fan-Theobald-von Neumann system \cite{gowda2019optimizing} is a triple $(\V,\W,\lambda)$, where $\V$ and $\W$ are  real inner product spaces and $\lambda:\V\to \W$ is a norm-preserving map satisfying a Fan-Theobald-von Neumann type inequality together with a condition for equality.  Examples include Euclidean Jordan algebras, systems induced by certain hyperbolic polynomials, and normal decomposition systems (Eaton triples).  The present article is a continuation of \cite{gowda2023commutativity} where the concepts of commutativity, automorphisms, majorization, and reduction were introduced and elaborated. Here, we describe some transfer principles and present Fenchel conjugate and subdifferential formulas.
\end{abstract}

\vspace{1cm}



\gap

{\it We dedicate this paper to Henry Wolkowicz, University of Waterloo, Canada, on the occasion
of his 75th birthday; We wish him a long, healthy, and productive life.
}
\section{Introduction}

Consider two real inner product spaces $\V$ and $\W$ with a (nonlinear) map $\lambda:\V\to \W$. For each $u\in \V$, let $[u]=\{x\in \V: \lx=\lu\}$. We say that the triple $(\V, \W, \lambda)$ is a Fan-Theobald-von Neumann system (FTvN system, for short) \cite{gowda2019optimizing} if 
\begin{equation} \label{eq: intro ftvn}
	\max \Big\{\! \ip{c}{x} : \, x \in [u] \Big\} = \ip{\lc}{\lu} \quad (\forall c, u \in \V).
\end{equation}
The inequality 
\[ \ip{x}{y} \leq \ip{\lx}{\ly} \quad (x, y \in \V), \]
which comes from \eqref{eq: intro ftvn} will be called {\it Fan-Theobald-von Neumann inequality} and the equality
\[ \ip{x}{y} = \ip{\lx}{\ly} \]
defines the {\it commutativity} of $x$ and $y$ in this system. 

\smallgap

Examples of FTvN systems abound. Given a real inner product space $\V$, the triple $(\V,\R, \lambda)$ with $\lx=||x||$ is a FTvN system in which the  Fan-Theobald-von Neumann inequality reduces to the Cauchy-Schwarz inequality. Taking $\V=\Rn=\W$ and $\lx=x^\downarrow$ (the decreasing rearrangement of $x$), we get a FTvN system where the   Fan-Theobald-von Neumann inequality reduces to the Hardy-Littlewood-Polya rearrangement inequality.
When $\V=\Sn$ (the space of all real $n\times n$ symmetric matrices with trace inner product) and $\W=\Rn$ with $\lambda(X)$ denoting the vector of eigenvalues of $X\in \Sn$ written in the decreasing order, we obtain the FTvN system $(\Sn,\Rn,\lambda)$, where the FTvN inequality reduces to the Ky Fan's inequality $\ip{X}{Y} \leq \ip{\lambda(X)}{\lambda(Y)}$ with  the corresponding equality case characterized by Theobald \cite{theobald1975inequality}. Considering the case of $\V=M_n$ (the space of all $n\times n$ complex matrices), $\W = \Rn$ and $\lambda(X) = s(X)$ (the vector of singular values of $X$ written in the decreasing order), one obtains a FTvN system where the FTvN inequality reduces to that of von Neumann.  Other examples include \cite{gowda2019optimizing}:

\begin{itemize}
	\item [$(a)$] The triple $(\V, \Rn, \lambda)$, where $\V$ is a Euclidean Jordan algebra of rank $n$ carrying the trace inner product with $\lambda : \V \to \Rn$ denoting the eigenvalue map,
	
	\item [$(b)$] The triple $(\V, \Rn, \lambda)$, where $\V$ is a finite dimensional real vector space and $p$ is a real homogeneous polynomial of degree $n$ that is  hyperbolic with respect to a vector $e\in \V$, complete and isometric, with $\lx$ denoting the vector of roots of the univariate polynomial $t \to p(te-x)$ written in the decreasing order, and
	
	\item [$(c)$] The triple $(\V, \W, \gamma)$ where $(\V, \G, \gamma)$ is a normal decomposition system (in particular, an Eaton triple) and $\W := \operatorname{span}(\gamma(\V))$.
\end{itemize}

Motivated by optimization considerations, FTvN systems were introduced in \cite{gowda2019optimizing} to transform linear/distance optimization problems
over certain sets in $\V$ (of the form $E=\lambda^{-1}(Q)$ - called spectral sets)  to problems over sets in $\W$. For example, it was shown in  \cite{gowda2019optimizing}, Section 3.1 that in a FTvN system $(\V,\W\,\lambda)$, for any $c\in \V$, $\phi:\W\rightarrow \R$, and any spectral set $E$ in $\V$,
\begin{equation}\label{eq: linear sup}
\sup_{x\in E} \Big\{\! \ip{c}{x} + (\phi \circ \lambda)(x) \Big\} = \sup_{u \in \lambda(E)} \Big\{\! \ip{\lc}{u} + \phi(u) \Big\}
\end{equation} 
with attainment of one supremum implying the attainment of the other and additionally implying a commutativity relation. 
In \cite{gowda2022commutation}, certain commutation principles were formulated and described in the setting of FTvN systems. A detailed analysis of the concepts of commutativity, automorphisms, majorization, and reduction in  Fan-Theobald-von Neumann systems was carried out in \cite{gowda2023commutativity}.

\smallgap

In the present paper, we focus on  FTvN systems $(\V, \W, \lambda)$ that come with an associated {\it reduced system} $(\W, \W, \mu)$; by definition, the two FTvN systems are related by the conditions $\operatorname{ran} \, \mu \subseteq \operatorname{ran} \, \lambda$ and $\mu \circ \lambda = \lambda$; see  \cite{gowda2023commutativity} for some examples and properties of FTvN systems with associated reduced systems. 

\smallgap

In the first part of the paper, we consider transfer principles dealing with the invariance of certain topological/convexity properties and/or operations. While such principles have been extensively studied in the context of Euclidean Jordan algebras \cite{baes2007convexity,sun2008lowner,lourencco2020generalized,jeong2016spectral,jeong2017spectral,gowda2018connectedness}, our goal here is to present them in the broader context of FTvN systems.
In a FTvN system, we specifically describe statements of the form 
\[ \lambda^{-1}(Q^\diamond)= \big( \lambda^{-1}(Q) \big)^\diamond, \]
where $\diamond$ is a topological/convexity operation such as the closure, interior, convex hull, etc. We also formulate a generalization of the celebrated result of Davis \cite{davis1957all} relating the convexity of $\phi\circ \lambda$ with that of $\phi$, where $\phi:\W\to \R$.
The second part of the paper is devoted to the study of the Fenchel conjugate and subdifferential of $\phi\circ \lambda$.  In the setting of a FTvN system, for a spectral set $S$, we derive the (Fenchel conjugate) formula 
\[ (\phi \circ \lambda)_{S}^{\ast}(z) = \phi_{\lambda(S)}^{\ast}\big( \lambda(z) \big), \]
which happens to be equivalent to \eqref{eq: intro ftvn} as well  as to (\ref{eq: linear sup}). Regarding subdifferentials, we show that 
\[ y \in \partial_{S} \Phi(\barx) \Longleftrightarrow \ly \in \partial_{\lambda(S)} \phi \big( \lambda(\barx) \big) \text{ and $y$ commutes with $\barx$}, \]
which also happens to be equivalent to \eqref{eq: intro ftvn}. 
These results generalize results of Lewis \cite{lewis1996convex} and Bauschke et al. \cite{bauschke2001hyperbolic} proved in the settings of normal decomposition systems and hyperbolic polynomials.

\smallgap

An outline of the paper is as follows: In Section 2, we cover some definitions, examples, and some known results. Section 3 deals with transfer principles. In Section 4, we describe the Fenchel conjugate of $\phi\circ \lambda$ and a subdifferential formula.

\section{Preliminaries}
Throughout this paper, we deal with  real inner product spaces with  $\ip{x}{y}$ denoting the inner product between two elements $x$ and $y$; we let  $\norm{x}$ denote the (induced) norm of $x$. In any such space, for a set $S$, we write $\overline{S}$, $S^\circ$, $\partial(S)$, $S^{c}$, and $S^\perp$ for the \textit{closure}, \textit{interior}, \textit{boundary}, \textit{(set-theoretic) complement}, and \textit{orthogonal complement} of $S$, respectively. We also write $\conv(S)$ (or $\conv S$), $\overline{\conv}\,S$, $\ext(S)$, and $\spn(S)$ for the \textit{convex hull}, \textit{closed convex hull}, \textit{set of all extreme points}, and \textit{span} of $S$, respectively. Throughout, $\R^n$ denotes the  real $n$-dimensional Euclidean space carrying the standard inner product.

\smallgap

We recall the  expanded version of the definition of a FTvN system.

\begin{definition}[{\bf FTvN system}, \cite{gowda2019optimizing}] \label{def: ftvn system}
	A \textit{Fan-Theobald-von Neumann system} (\textit{FTvN system}, for short) is a triple $(\V, \W, \lambda)$, where
	$\V$ and $\W$ are  real inner product spaces and $\lambda:\V\to \W$ is a map satisfying the following conditions:
	\begin{itemize}
		\item [$(A1)$] $\norm{\lx} = \norm{x}$ for all $x \in \V$.
		\item [$(A2)$] $\ip{x}{y} \leq \ip{\lx}{\ly}$ for all $x, y \in \V$.
		\item [$(A3)$] For any $c \in \V$ and $q \in \lambda(\V)$, there exists $x \in \V$ such that
		\begin{equation}\label{A3}
			\lx = q \quad \text{and} \quad \ip{c}{x} = \ip{\lc}{\lx}.
		\end{equation}
	\end{itemize}
\end{definition}

{\it It has been observed in \cite{gowda2023commutativity} that conditions $(A1)$--$(A3)$ are equivalent to \eqref{eq: intro ftvn}.
}

\smallgap

Let $(\V, \W, \lambda)$ be a FTvN system. The map $\lambda$ will be called the \textit{eigenvalue map}. We denote the \textit{range} of $\lambda$ by $\ran \, \lambda$; the \textit{$\lambda$-orbit of an element $x \in \V$} is defined by
\[ [x] := \{ y \in \V : \ly = \lx \}. \]
More generally, for a set $S$ in $\V$, the \textit{$\lambda$-orbit of $S$} is
\[ [S] := \bigcup_{x \in S} [x]. \]

\smallgap

A set $E$ in $\V$ is said to be a \textit{spectral set} if it is of the form $E = \lambda^{-1}(Q)$ for some $Q \subseteq \W$, or equivalently, a union of $\lambda$-orbits. It is easy to see that a set $E$ is a spectral set if and only if the implication $x \in E \Rightarrow [x] \subseteq E$ holds. For any set $S$ in  $\V$, $[S]$ is a spectral set; we call it the {\it spectral hull} of $S$. Also, the set
\[ \core{S} := [S^c]^c, \]
being the complement of a spectral set,
is a spectral set; we will call this, the {\it spectral core} of $S$. Note that $[S]$ is the smallest spectral set containing $S$, while $\core{S}$ is the largest spectral set contained in $S$. Moreover, $S$ is spectral if and only if $\core{S} = [S].$

\smallgap

A real-valued function $\Phi : \V \to \R$ is a {\it spectral function} if it is of the form $\Phi = \phi \circ \lambda$ for some function $\phi : \W \to \R$, or equivalently, $\Phi$ is a constant on every $\lambda$-orbit. Note that $\phi$ need be defined only on $\lambda(\V)$. Also, in Section 4, while discussing Fenchel conjugate and subdifferentials, we allow $\phi$ and $\Phi$ to be extended real-valued functions.

\smallgap

The following result describes how spectral sets and spectral functions are related.

\begin{proposition} \label{prop: epigraph} 
	Let $(\V, \W, \lambda)$ be a FTvN system. Then the following hold.
	\begin{itemize}
		\item[$(a)$] A set $E$ in $\V$ is spectral if and only if its indicator function $\mathbf{1}_{E} : \V \to \R$ given by 
		\[ \mathbf{1}_{E}(x) = 
		\begin{cases} 
			1 & \text{ if } x \in E,   \\
			0 & \text{ if } x \notin E
		\end{cases} \]
		is a spectral function on $\V$.
		
		\item[$(b)$] A function $\Phi : \V \to \R$ is spectral if and only if its epigraph given by
		\[ \epi \Phi = \big\{ (t, x) \in \R \times \V : t \geq \Phi(x) \big\} \]
		is a spectral set in the (product) FTvN system $(\R \times \V, \R \times \W, \Lambda),$ where $\Lambda \big( (t, x) \big) = \big(t, \lx \big)$.
	\end{itemize}
\end{proposition}

\begin{proof}
	$(a)$ Suppose $E$ in $\V$ is a spectral set, i.e., $E = \lambda^{-1}(Q)$ for some $Q\subseteq \W$. Then it is easy to verify that $\mathbf{1}_{E} = \mathbf{1}_{Q} \circ \lambda$; thus $\mathbf{1}_{E}$ is a spectral function.
	
	Conversely, suppose the indicator function $\mathbf{1}_{E} : \V \to \R$ is spectral so that $\mathbf{1}_{E} = \phi \circ \lambda$ for some $\phi : \W \to \R$. Define
	\[ Q = \{ u \in \W : \phi(u) = 1 \}. \]
	We now show $E = \lambda^{-1}(Q)$, proving that $E$ is spectral. To see this, take $x \in E$. Since $1 = \mathbf{1}_{E}(x) = \phi \big( \lx \big)$, we have $\lx \in Q$. Thus, $x \in \lambda^{-1}(Q)$. For the reverse implication, take $x \in \lambda^{-1}(Q)$. Then $\lx \in Q$, hence $\mathbf{1}_{E}(x) = \phi \big( \lx \big) = 1$, implying $x \in E$. Consequently, $E = \lambda^{-1}(Q)$.
	
	\smallgap
	
	$(b)$ It is easy to see that $(\R \times \V, \R \times \W, \Lambda)$ with $\Lambda \big( (t, x) \big) = \big( t, \lx \big)$ is a FTvN system. Given a spectral function $\Phi : \V \to \R$ such that $\Phi = \phi \circ \lambda$ for some $\phi : \W \to \R$, we show $\epi \Phi = \Lambda^{-1}(\epi \phi)$. Indeed, we have
	\begin{align*}
		(t, x) \in \epi \Phi 
		& \iff t \geq \Phi(x) \\
            & \iff  t \geq  \phi \big( \lx \big) \\
		& \iff \Lambda \big( (t, x) \big) = \big( t, \lx \big) \in \epi \phi \\
		& \iff (t, x) \in \Lambda^{-1}(\epi \phi).
	\end{align*}
	This shows that $\epi \Phi$ is a spectral set in $(\R \times \V, \R \times \W, \Lambda)$.
	
	For the converse, suppose $\epi \Phi$ is a spectral set in $(\R \times \V, \R \times \W, \Lambda)$ so that  $\epi \Phi = \Lambda^{-1}(Q)$ for some $Q$ in $\R \times \W$. Then, for each $u\in \lambda(\V)$, the set $\{t\in \R: (t,u)\in Q\}$ is bounded below. We now define $\phi : \lambda(\V) \to \R$  by
	\[ \phi(u):= \inf \, \{ t \in \R : (t, u) \in Q \} \]
	and extend it to $\W$ arbitrarily. 
	Then, for $x \in \V$, we let  $u = \lx \in \lambda(\V)$ so that
	\begin{align*}
		\phi(\lx) 
		& = \inf \big\{ t \in \R : \big( t, \lx \big) \in Q \big\} \\
		& = \inf \big\{ t \in \R : \Lambda \big( (t, x) \big) \in Q \big\} \\
		& = \inf \big\{ t \in \R : (t, x) \in \Lambda^{-1}(Q) = \epi \Phi \big\} \\
		& = \Phi(x).
	\end{align*}
	As $x \in \V$ is arbitrary, it follows that $\Phi = \phi \circ \lambda$. \qed
\end{proof}

For ease of reference, we now recall  some definitions, results, and examples from  earlier works.

\begin{definition}[{\bf Commutativity and majorization}]
Let $(\V, \W, \lambda)$ be a FTvN system and $x,y\in \V$. Relative to this system, we say that 
\begin{itemize}
     \item [$(a)$] 
 $x$ and $y$  {\it commute }if
$ \ip{x}{y} = \ip{\lx}{\ly}$ and 
\item [$(b)$]  $x$ is {\it majorized} by $y$ and write $x\prec y$ if $x\in \conv \,[y]$.
\end{itemize}
\end{definition}

\begin{proposition}[\cite{gowda2019optimizing}, Section 2] \label{prop: basic ftvn theorem}
    Let $(\V, \W, \lambda)$ be a FTvN system. Then, the following hold for $x, y, c \in \V$:
	\begin{itemize}
		\item [$(a)$] $\lambda(tx) = t\lx$ for all $t \geq 0$.
		
		\item [$(b)$] $\norm{\lx - \ly} \leq \norm{x-y}$.
		
		\item [$(c)$] $\ip{\lc}{\lambda(x+y)} \leq \ip{\lc}{\lx} + \ip{\lc}{\ly}$. More generally, for $c, x_1, x_2, \ldots, x_k$ in $\V$,
		\begin{equation} \label{general sublinearity}
			\Big\langle \lc, \lambda(x_1 + x_2 + \cdots + x_k) \Big\rangle \leq \Big\langle \lc, \lambda(x_1) + \lambda(x_2) + \cdots + \lambda(x_k) \Big\rangle.
		\end{equation}
	
		\item [$(d)$] $F := \operatorname{ran} \, \lambda$ is a convex cone in $\W$. It is closed if $\V$ is finite dimensional.
		
		\item [$(e)$] The following are equivalent:
		\begin{itemize}
			\item [$(i)$] $x$ and $y$ commute in $(\V, \W, \lambda)$, that is, $\ip{x}{y} = \ip{\lx}{\ly}$.
			\item [$(ii)$] $\lambda(x+y) = \lx + \ly$.
			\item [$(iii)$] $\norm{\lx - \ly} = \norm{x-y}$.
		\end{itemize}
	\end{itemize}
\end{proposition}

\begin{proposition}[\cite{gowda2023commutativity}, Proposition 4.2]
    Suppose $(\V,\W,\lambda)$ is a FTvN system. If $E$ is convex and spectral in $\V$, then $\lambda(E)$ is convex in $\W$.
\end{proposition}

Recall that for a set $E$ in $\V$, its \textit{polar} and \textit{dual} are defined respectively by
\[ E^p := \{x \in \V : \ip{x}{y} \leq 0 \text{ for all } y \in E\}, \quad E^* = -E^p. \]

\begin{proposition}[\cite{gowda2023commutativity}, Proposition 4.3] \label{prop: spectral invariance} 
    Let $E$ be a spectral set in a FTvN system $(\V, \W, \lambda)$. Then the following statements hold:
	\begin{itemize}
		\item [$(a)$] $\overline{E}$, $E^\circ$, and $\partial(E)$ are  spectral.
		\item [$(b)$] If $\V$ is a Hilbert space, then $\overline{\conv}\,E$ is a spectral set.
		\item [$(c)$] If $\V$ is finite dimensional, then $\conv E$ is a spectral set.
		\item [$(d)$] If $\V$ is a Hilbert space, then $E^p$ is a spectral set. In particular, if $\V$ is a Hilbert space and $S$ is a spectral set which is also a subspace in $\V$, then, $S^\perp$ is spectral. 
    	\item [$(e)$] If $\V$ is a Hilbert space, then the sum of two compact convex spectral sets in $\V$ is spectral. 
    	\item [$(f)$] If $\V$ is finite dimensional, then the sum of two convex spectral sets is spectral.
	\end{itemize}
\end{proposition}

\begin{proposition}[\cite{gowda2023commutativity}, Corollary 5.5] \label{prop: convex hull of a sum}
    Consider a FTvN system $(\V, \W, \lambda)$, where $\V$ is finite dimensional. Then, for all $a,b\in \V$,
    \[ \conv [a+b] \subseteq \conv [a] + \conv [b]. \]
\end{proposition}
We recall the definition of a reduced system.

\begin{definition}[{\bf Reduced system}] \label{def: reduced system}
    Let $(\V,\W,\lambda)$ be a FTvN system. Suppose $(\W,\W,\mu)$ is a FTvN system such that
	\begin{itemize}
		\item [$(C1)$] $\mu \circ \lambda = \lambda$, and
		
		\item [$(C2)$] $\operatorname{ran} \, \mu \subseteq \operatorname{ran} \, \lambda$.
	\end{itemize}
	Then, we will say that $(\W,\W,\mu)$ is a \textit{reduced system} of $(\V, \W, \lambda)$.
\end{definition}

Let $(\W, \W, \mu)$ be a reduced system of a FTvN system $(\V,\W,\lambda)$. For any $w \in \W$, by $(C2)$, we may choose $x\in \V$ such that $\mu(w)=\lx$. Then $(C1)$ gives $\mu \big( \mu(w) \big) = \mu \big( \lx \big) = \lx = \mu(w)$. This implies that $\mu^2 = \mu$ on $\W$. Also, from $(C1)$, $\mbox{ran}\,\mu=\mbox{ran}\,\lambda$. Thus, 
\begin{center}
    {\it when $(\W, \W, \mu)$ is a reduced system of $(\V, \W, \lambda)$, we have $\mu^2 = \mu$ and $\mbox{ran}\,\mu = \mbox{ran}\,\lambda$.
}
\end{center}

\begin{proposition}[\cite{gowda2023commutativity}, Theorem 9.3] \label{prop: Lidskii type theorem} 
	Suppose $(\W, \W, \mu)$ is a reduced system of $(\V, \W, \lambda)$ with $\W$ finite dimensional. Let $F := \operatorname{ran} \, \lambda$ and $F^*$ denote the dual of the cone $F$ in $\W$. Then,
	the following statements hold:
	\begin{itemize}
		\item [$(a)$] If $u, v \in F$ with $u-v \in F^*$, then $v \prec u$ in $\W$.
		
		\item [$(b)$] For $x_1, x_2, \ldots, x_k \in \V$, $\lambda(x_1 + x_2 + \cdots + x_k) \prec \lambda(x_1) + \lambda(x_2) + \cdots + \lambda(x_k)$ in $\W$.
		
		\item [$(c)$] $x \prec y$ in $\V$ implies $\lx \prec \ly$ in $\W$. The converse holds if $\V$ is finite dimensional.
	\end{itemize}
\end{proposition}

We now present some examples of FTvN systems with their corresponding reduced systems. For more examples, we refer to \cite{gowda2023commutativity}.

\begin{example}
	Consider the FTvN system $(\V, \R, \lambda)$, where $\V$ is an inner product space and $\lx := \norm{x}$ for all $x \in \V$. Then, $(\R, \R, \mu)$ with $\mu(r) = |r|$ for $r \in \R$, is a reduced system of $(\V, \R, \lambda)$. However, the FTvN system $(\R, \R, \nu)$ with $\nu(r) = r$ is {\it not} a reduced system of $(\V, \R, \lambda)$ because condition $(C2)$ in the above definition fails to hold. 
\end{example}

\begin{example}
	Suppose $\V$ is a Euclidean Jordan algebra of rank $n$ (with the trace inner product). Then, $(\Rn, \Rn, \mu)$ with $\mu(q) = q^\downarrow$ on $\Rn$ is a reduced system of the FTvN system $(\V, \Rn, \lambda)$, where $\lx$ is the eigenvalue vector of $x \in \V$. We note that in the FTvN system $(\Rn, \Rn, \mu)$ a set is spectral if and only if it is invariant under permutation matrices. Such sets are traditionally called {\it symmetric} sets.
\end{example}

\begin{example}
	Suppose $(\W, \W, \mu)$ is an FTvN system with $\mu^2 = \mu$. Then, $(\W, \W, \mu)$ is a reduced system of itself. In particular, every normal decomposition system has this property. (See the Appendix in \cite{gowda2023commutativity} for the definition of a normal decomposition system.)
\end{example}

\begin{example} \label{eg: product FTvN system}
    Consider two FTvN systems $(\V_1, \W_1, \lambda_1)$ and $(\V_2, \W_2, \lambda_2)$ with their corresponding reduced systems $(\W_1, \W_1, \mu_1)$ and $(\W_2, \W_2, \mu_2)$. Then, see Example 2.14 in \cite{gowda2023commutativity}, their Cartesian product can be made into a FTvN system $(\V_1 \times \V_2, \W_1 \times \W_2, \Lambda)$, where
	\[ \Lambda \big( (x_1, x_2) \big) = \big( \lambda_1(x_1), \lambda_2(x_2) \big) \;\; \text{for} \;\; (x_1, x_2) \in \V_1 \times \V_2. \]
	In this case, the triple $(\W_1 \times \W_2, \W_1 \times \W_2, M)$, with 
	\[ M \big( (u_1, u_2) \big) = \big( \mu_1(u_1), \mu_2(u_2) \big) \;\; \text{for} \;\; (u_1, u_2) \in \W_1 \times \W_2 \]
	becomes the corresponding reduced system. Indeed, since $\operatorname{ran} \mu_i \subseteq \operatorname{ran} \lambda_i$ for $i = 1, 2$, it is easy to see that $\operatorname{ran} M \subseteq \operatorname{ran} \Lambda$, proving $(C1)$. Also, using the fact that $\mu_i \circ \lambda_i = \lambda_i$ for $i = 1, 2$, we have
	\begin{align*}
	    M \big( \Lambda(x_1, x_2) \big) 
     &= M \big( \lambda(x_1), \lambda(x_2) \big)
     = \big( \mu(\lambda(x_1)), \mu(\lambda(x_2)) \big) \\
     &= \big( \lambda(x_1), \lambda(x_2) \big) 
     = \Lambda \big( (x_1, x_2) \big),
	\end{align*}
	for any $(x_1, x_2) \in \V_1 \times \V_2$. Hence $M \circ \Lambda = \Lambda$, justifying $(C2)$. 
	
	In particular, for a FTvN system $(\V, \W, \lambda)$ with its reduced system $(\W, \W, \mu)$, the triple $(\R \times \V, \R \times \W, \Lambda)$, where $\Lambda(t, x) = \big( t, \lx \big)$, is a FTvN system and $(\R \times \W, \R \times \W, M)$ with $M(t, u) = \big( t, \mu(u) \big)$ is the corresponding reduced system.
	
	As an illustrative example, consider the triple $(\R^{n+1}, \R^2, \lambda)$ with
	\[ \lambda \big( (t, x) \big) = \big( t, \norm{x}_2 \big) \;\; \text{for} \;\; (t, x) \in \R \times \R^n. \]
	It is easy to see that it is the Cartesian product of two FTvN systems $(\R, \R, \operatorname{Id})$ and $(\R^n, \R, \norm{\cdot}_2)$; hence a FTvN system. Here, $(\R^2, \R^2, \mu)$ with $\mu \big( (t, s) \big) = (t, \abs{s})$ is the reduced system. Now, corresponding to the spectral set $Q = \big\{ (t, s) \in \R^2 : t \geq \abs{s} \big\}$ in $(\R^2, \R^2, \mu)$, we get the second-order cone
	\[\lambda^{-1}(Q) = \big\{ (t, x) \in \R \times \R^n : t \geq \norm{x}_2 \big\}. \]
 
\end{example}

\section{Transfer principles}

Consider a FTvN system $(\V, \W, \lambda)$. Recall that a set in $\V$ is a spectral set if it is of the form $\lambda^{-1}(Q)$ for some $Q \subseteq \W$. Also, a function $\Phi : \V \to \R$ is a spectral function if it is of the form $\phi \circ \lambda$ for some function $\phi : \lambda(\V) \to \R$ (without loss of generality, we may let $\phi : \W \to \R$). Motivated by various transfer principles in the setting of Euclidean Jordan algebras and normal decomposition systems \cite{jeong2016spectral,jeong2017spectral,lewis1996convex}, we a raise basic question:
{\it Which topological/convexity properties of $Q$ and $\phi$ are (respectively) carried over to $E = \lambda^{-1}(Q)$ and $\Phi = \phi \circ \lambda$?}
In this section, we formulate several results addressing this and related questions.

\begin{proposition} \label{prop: basic prop}
    Let $(\V, \W, \lambda)$ be a FTvN system, $Q\subseteq \W$, and $E:=\lambda^{-1}(Q)$. Then the following statements hold.
    \begin{itemize}
        \item [$(a)$]
    If $Q$ is open (closed) in $\W$, then $E$ is open (respectively, closed) in $\V$.
    \item [$(b)$] If $Q$ is compact in $\W$ and $\V$ is finite dimensional, then $E$ is compact in $\V$.
    \item [$(c)$] If $\phi : \W \to \R$ is continuous, then $\Phi = \phi \circ \lambda$ is continuous.
    \end{itemize}
\end{proposition}

\begin{proof}
	The first statement follows from the continuity of $\lambda$. The second one follows from the norm-preserving property of $\lambda$ (that $\norm{\lx} = \norm{x}$ for all $x$) and the finite dimensionality of $\V$. Finally, the continuity of $\Phi$ comes from the continuity of $\phi$ and $\lambda$. \qed
\end{proof}

\begin{corollary}
    Suppose $(\V,\W,\lambda)$ be a FTvN system, where $\V$ is finite dimensional. Consider a set  $S$ in $\V$ with its spectral hull $[S]$ and spectral core $\core{S}$. Then the following statements hold.
    \begin{itemize}
        \item [$(i)$] If $S$ is closed, then $\lambda(S)$ is closed in $\W$.
        \item [$(ii)$] If $S$ is closed, then  $[S]$ is closed.
        \item [$(iii)$] If $S$ is open, then $\core{S}$ is open.
        \item [$(iv)$] If $S$ is compact, then $[S]$ is compact.
    \end{itemize}
\end{corollary}

\begin{proof}
$(i)$ Assume that $S$  is closed. Consider a sequence $(u_k)$ in $\lambda(S)$ with $u_k \to u \in \W$. Let $u_k = \lambda(x_k)$ with $x_k \in S$ for each $k$. As $\lambda$ is  norm-preserving and $(u_k)$ is bounded, we see that $(x_k)$ is also bounded. Since $\V$ is finite dimensional, without loss of generality, we may assume $x_k \to x$ for some $x \in S$. Then, by the continuity of $\lambda$, we have $u_k = \lambda(x_k) \to \lambda(x) \in \lambda(S)$. Thus, $u \in \lambda(S)$, proving the closedness of $\lambda(S)$.

\smallgap

$(ii)$ From the above item, $\lambda(S)$ is closed in $\W$. Then Proposition \ref{prop: basic prop} implies that $[S] = \lambda^{-1} \big( \lambda(S) \big)$ is also closed.

\smallgap

$(iii)$ Suppose $S$ is open. As $S^c$ is closed, by the above item, we see that $[S^c]$ is closed; hence $[S^c]^{c}$ is open. Since, by definition, $\core{S} = [S^c]^c$, we get the stated assertion.

\smallgap

$(iv)$ Now suppose that $S$ is compact. Then, by the continuity of $\lambda$, $\lambda(S)$ is compact in $\W$. Then, from the above Proposition \ref{prop: basic prop}$(c)$, we see that $[S] = \lambda^{-1}\big( \lambda(S) \big)$ is also compact. \qed
\end{proof}

\begin{remark} 
In general, the connectedness/convexity properties of a set need not be carried over to its spectral hull (inverse image). For example, in the FTvN system $(\R^2, \R^2, \mu)$ with $\mu(u) = u^\downarrow$, one can take a connected convex set $S = \{(2, 1)\}$ and see that $[S] = \mu^{-1}(S) = \{(1,2), (2,1)\}$, which is neither connected nor convex. However, in certain settings, we can show that if $S$ is open (connected, arcwise connected), then $[S]$ is open (respectively, connected, arcwise connected). For example, consider the system $(\V, \Rn, \lambda)$, where $\V$ is a simple Euclidean Jordan algebra of rank $n$ that carries the trace inner product. Then, any Jordan frame in $\V$ can be mapped onto any other by an automorphism of $\V$. Specifically, for any $u \in \V$,
\[ [u] = \{Au: A \in G \}, \]
where $G$ is the connected component of the identity transformation in the automorphism group of $\V$ (\cite{gowda2018connectedness}, Proposition 2.2). So, $[S] = \bigcup_{A\in G} A(S)$. Now, when $S$ is open, each $A(S)$ is open and so $[S]$, being the union of open sets, is also open. Now, suppose $S$ is connected. Then, $[S]$ is the union of $S$ and connected sets $[s]$ with $s$ varying over $S$. Since $[s] \cap S \neq \emptyset$ for every $s \in S$, we see that the above union is also connected. Thus, $[S]$ is connected. Finally, when $S$ is arcwise connected, we apply Theorem 3.1 in \cite{gowda2018connectedness} to see that $[s]$ is arcwise connected for every $s \in S$ and that $[S]$ is also arcwise connected. 
\end{remark}

The following example shows that, generally, the closure, interior, and convexity properties do not behave well under inverse images.

\begin{example} \label{eg: no spectrality}
Consider the FTvN system $(\R^2,\R,\lambda)$, where $\lx := \norm{x}$ for all $x\in \R^2$. Then the following are easy to verify:
\begin{itemize}
    \item [$(i)$] The interval $Q=(-1,0]$ is not closed in $\R$, but the set $\lambda^{-1}(Q)$ is closed in $\R^2$. Similarly, the interval $Q = [0, 1)$ is not open in $\R$, while $\lambda^{-1}(Q)$ is open in $\R^2$. 
    
    \item [$(ii)$] For the interval $Q = (-1, 0)$ in $\R$, we have $\overline{\lambda^{-1}(Q)} \neq \lambda^{-1}(\, \overline{Q}\, )$.
    
    \item [$(iii)$] For the interval $Q = [0, 1]$ in $\R$, we have $\lambda^{-1}(Q)^{\,\circ} \neq \lambda^{-1}(Q^{\circ})$, $\lambda^{-1}\big( \partial(Q) \big) \neq \partial \big( \lambda^{-1}(Q) \big)$, and $\lambda^{-1}\big( \ext(Q) \big) \neq \ext \big( \lambda^{-1}(Q) \big)$.
    
    \item [$(iv)$] For the compact convex set $Q = \{1\}$ in $\R$, $\lambda^{-1}(Q)$ is not convex in $\R^2$ and $\conv \lambda^{-1}(Q) \neq \lambda^{-1}(\conv Q)$.
\end{itemize}
\end{example}

Note that in the above example, $(\R, \R, \mu)$ with $\mu(r) = \abs{r}$ is a reduced system of $(\R^2, \R, \lambda)$, but the considered sets are not spectral in $(\R, \R, \mu)$. As we see below, positive results are obtained when one works with spectral sets in a reduced system.

\smallgap

Let  $(\W, \W, \mu)$ be a reduced system of $(\V, \W, \lambda)$. {\it Recall that a set in $\W$ is spectral if it is so in $(\W, \W, \mu)$}. The $\mu$-orbit of an element in $\W$ is denoted by the same bracket notation that we use in $\V$.

\smallgap

As a prelude to our positive results, we present a  technical result. 

\begin{proposition} \label{prop: description of spectrality}
Suppose $(\W, \W, \mu)$ is a reduced system of $(\V, \W, \lambda)$. Then the following statements hold:
\begin{itemize}
    \item [$(a)$] If $E = \lambda^{-1}(Q)$ with $Q \subseteq \W$, then $\lambda(E) = Q \cap \mu(Q)$. If $Q$ is spectral in $\W$, then $\mu(Q) \subseteq Q$ and $\lambda(E) = \mu(Q)$.
    
    \item [$(b)$] Every spectral set in $\V$ can be written as the $\lambda$-inverse image of a spectral set in $\W$. In fact, for any $Q \subseteq \W$, the set $\widetilde{Q} := [Q \cap \mu(Q)]$ is spectral in $\W$ and
    \[ \lambda^{-1}(Q) = \lambda^{-1}(\widetilde{Q} ). \]
    
    \item [$(c)$] If $Q$ is spectral in $\W$, then $Q = [\mu(Q)]$ and
    \[ \lambda^{-1}(Q) = \lambda^{-1} \big( \mu(Q) \big). \]

    \item [$(d)$] Every spectral function on $\V$ can be written as the composition of a spectral function on $\W$ and $\lambda$. 
\end{itemize}
\end{proposition}

\begin{proof}
$(a)$ Consider $E = \lambda^{-1}(Q)$, where $Q \subseteq \W$. For any $x \in E$, $q := \lx \in Q$; so, by $(C1)$ (of Definition \ref{def: reduced system}), $\lx = \mu \big( \lx \big) = \mu(q) \in \mu(Q)$. Hence, $\lambda(E) \subseteq Q \cap \mu(Q)$. To see the reverse inclusion, let $p \in Q \cap \mu(Q)$. Then, $p \in Q$ and $p = \mu(q)$ for some $q \in Q$. By $(C2)$, we can write $p = \mu(q) = \lx$ for some $x \in \V$. Clearly, $x \in E$ and $p \in \lambda(E)$. Hence $\lambda(E) = Q \cap \mu(Q)$.

Now suppose $Q$ is spectral in $\W$. Since $\mu^2 = \mu$, for any $u \in Q$, we have $\mu(\mu(u)) = \mu(u)$. So, $\mu(u)$ is in the $\mu$-orbit of $u$. As $Q$ is a spectral set in $\W$, we have $\mu(u) \in Q$. This proves that $\mu(Q) \subseteq Q$. Then the equality $\lambda(E) = Q \cap \mu(Q) = \mu(Q)$ follows.

\smallgap

$(b)$ Consider a spectral set $E = \lambda^{-1}(Q)$ with $Q \subseteq \W$. Clearly, $\lambda^{-1} \big( Q \cap \mu(Q) \big) \subseteq \lambda^{-1}(Q)$. On the other hand, if $x \in \lambda^{-1}(Q) = E$, then $\lx \in \lambda(E) = Q \cap \mu(Q)$ by Item $(a)$. Hence,  $\lambda^{-1}(Q) \subseteq \lambda^{-1} \big( Q \cap \mu(Q) \big)$. Thus,
\[ \lambda^{-1}(Q) = \lambda^{-1} \big( Q \cap \mu(Q) \big). \]

We now claim
\[ \lambda^{-1}(Q \cap \mu(Q)) = \lambda^{-1} \big( [Q \cap \mu(Q)] \big). \]
Let $y \in \lambda^{-1} \big( [Q \cap \mu(Q)] \big)$ so that $\ly = p \in [q]$ for some $q \in Q \cap \mu(Q)$. Then, from $(C1)$, $\ly = \mu \big( \ly \big) = \mu(p) = \mu(q)$. However, $q = \mu(r)$ for some $r \in Q$ and so, $\mu(q) = \mu^2(r) = \mu(r) = q$. It follows that $\ly = \mu(q) = q \in Q \cap \mu(Q)$. Hence,
\[ \lambda^{-1} \big( [Q \cap \mu(Q)] \big) \subseteq \lambda^{-1} \big( Q \cap \mu(Q) \big). \]
Since the reverse inclusion is obvious, we see that $E=\lambda^{-1}(Q) = \lambda^{-1}(\widetilde{Q} )$, where $\widetilde{Q}:=[Q\cap \mu(Q)]$ is spectral in $\W$. 

\smallgap

$(c)$ Suppose $Q$ is spectral in $\W$. By $(a)$, $\mu(Q) \subseteq Q$. Then, $[\mu(Q)] \subseteq [Q] = Q$. As observed before, for every $u \in Q$,  $u$ and $\mu(u)$ lie in the same $\mu$-orbit; hence, $Q \subseteq [\mu(Q)]$. Thus, $Q = [\mu(Q)]$. \\
 
Now we show that $\lambda^{-1}(Q) = \lambda^{-1} \big( \mu(Q) \big)$. From $\mu(Q) \subseteq Q$, we have $\lambda^{-1}(\mu(Q)) \subseteq \lambda^{-1}(Q)$. To see the reverse inclusion, let $x \in \lambda^{-1}(Q)$ so that $\lx \in Q$. Then, $\lx = \mu \big( \lx \big) \in \mu(Q)$. This proves that $x \in \lambda^{-1} \big( \mu(Q) \big)$. Thus, we have $\lambda^{-1}(Q) = \lambda^{-1} \big( \mu(Q) \big)$.

\smallgap

$(d)$ Consider a spectral function $\Phi$ on $\V$ that is written as $\phi \circ \lambda$  for some $\phi : \W \to \R$. Define $\widetilde{\smash[t]\phi} : \W \to \R$ by $\widetilde{\smash[t]\phi}(u) := \phi \big( \mu(u) \big)$ for any $u \in \W$. Clearly, $\widetilde{\smash[t]\phi}$ is constant on the $\mu$-orbits, hence a spectral function on $\W$. Additionally, if $u = \lx$ for some $x \in \V$, then 
\[ \widetilde{\smash[t]\phi} \big( \lx \big) = \phi \big( \mu \big( \lx \big) \big) = \phi \big( \lx \big) = \Phi(x), \]
where we have used condition $(C1)$ in Definition \ref{def: reduced system}. Hence $\Phi = \widetilde{\smash[t]\phi} \circ \lambda$, where $\widetilde{\smash[t]\phi}$ is a spectral function on $\W$. \qed
\end{proof}

We now come to first of several key results of this section. For results of this type in the settings of Euclidean Jordan algebras and normal decomposition systems, see \cite{jeong2017spectral,lewis1996convex}.

\begin{theorem} \label{thm: transfer principle1}
    Suppose $(\V,\, \W,\, \lambda)$ is a FTvN system with its reduced system $(\W,\, \W,\, \mu)$. Let $E = \lambda^{-1}(Q)$, where $Q$ is a spectral set in $\W$. Then
    \begin{itemize}
        \item[$(a)$] $\overline{E} = \overline{\lambda^{-1}(Q)} = \lambda^{-1}(\, \overline{Q} \,)$.
        
        \item[$(b)$] $E^{\circ} = \lambda^{-1}(Q)^{\,\circ} = \lambda^{-1}(Q^{\circ})$.
        
        \item [$(c)$] $\partial(E)=\lambda^{-1}\big( \partial (Q) \big).$
    \end{itemize}
\end{theorem}

\begin{proof}
	$(a)$ By the continuity of $\lambda$, $\overline{E} = \overline{\lambda^{-1}(Q)} \subseteq \lambda^{-1}(\, \overline{Q}\, )$. To see the reverse inclusion, let $x \in \lambda^{-1}(\, \overline{Q}\, )$ so that $\lx \in \overline{Q}$. Then, there exists a sequence $(u_k)$ in $Q$ such that $u_k \to \lx$. By the continuity of $\mu$, we have $\mu(u_k) \to \mu \big( \lx \big) = \lx$. Note that $\mu(u_k) \in \mu(\W) \subseteq \lambda(\V)$. Thus, by (A3) in Definition \ref{def: ftvn system}, for each $k$, there exists $x_k \in \V$ such that $\lambda(x_k) = \mu(u_k)$ with  $x_k$ and $x$ commuting. Since $\lambda$ is distance-preserving on two commuting elements (see Proposition \ref{prop: basic ftvn theorem}), we have
	\[ \norm{x_k - x} = \norm{\lambda(x_k) - \lx} = \norm{\mu(u_k) - \lx}. \]
	From Proposition \ref{prop: description of spectrality}$(a)$, $\mu(Q)\subseteq Q$. So,  $\lambda(x_k) = \mu(u_k) \in \mu(Q) \subseteq Q$. Hence, $x_n \in \lambda^{-1}(Q)$ and
	\[ \lim_{k \to \infty} \norm{x_k - x} = \lim_{k \to \infty} \norm{\mu(u_k) - \lx} = 0. \]
	This shows that $x \in \overline{\lambda^{-1}(Q)}$, implying $\lambda^{-1}(\, \overline{Q}\, ) \subseteq \overline{\lambda^{-1}(Q)}$.
	Hence we have Item $(a)$.
	\smallgap
	
	$(b)$ Since $Q$ is spectral in $\W$, $Q^c$ is also spectral in $\W$. Since $E^{\circ} = \big (\, \overline{E^{c}}\, \big )^c$, we see, by $(a)$, that
	\begin{align*}
	    E^{\circ} & = \Big (\lambda^{-1}(Q)\Big )^{\,\circ} 
     = \Big(\, \overline{\lambda^{-1}(Q)^c}\, \Big)^c 
     = \Big(\, \overline{\lambda^{-1}(Q^c)}\, \Big)^c \\
     &= \Big( \lambda^{-1} \Big (\, \overline{Q^c}\, \Big ) \Big)^c 
     = \lambda^{-1} \Big( \, \big (\overline{Q^c}\, \big)^c\, \Big) 
     = \lambda^{-1}(Q^{\circ}).
	\end{align*} 
	
	\smallgap
	
	$(c)$ As $\partial(E)=\overline{E}\,\backslash\,E^\circ$, the result follows from $(a)$ and $(b)$ and the set-theoretic properties of $\lambda^{-1}$. \qed
\end{proof}

Suppose $(\V,\W,\lambda)$ is a FTvN system. By the continuity of $\lambda$ we know (from  Proposition \ref{prop: basic prop}) that the inverse image of a closed set (open set) is closed (respectively, open). Also, the inverse image of a compact set is compact when $\V$ is finite dimensional. In the following result, we show that under appropriate conditions, a set $Q$ is closed (open, compact) in $\W$ if and only if its inverse image $\lambda^{-1}(Q)$ is closed (respectively, open, compact) in $\V$.

\begin{proposition}\label{closedness of Q}
    Suppose $(\V,\, \W,\, \lambda)$ is a FTvN system with its reduced system $(\W,\, \W,\, \mu)$. Let $E = \lambda^{-1}(Q)$, where $Q$ is a spectral set in $\W$. Then the following statements hold. 
    \begin{itemize}
         \item [$(i)$] Suppose $\V$ is finite dimensional. If $E$ is closed (open) in $\V$, then $Q$ is closed (respectively, open) in $\W$.
         \item [$(ii)$] Suppose both $\V$ and $\W$ are finite dimensional. If $E$ is compact in $\V$, then $Q$ is compact in $\W$.
    \end{itemize}
\end{proposition}

\begin{proof}
    $(i)$ Suppose $E$ is closed in $\V$;  let $(u_k)$ be a sequence in $Q$ such that $u_k \to u \in \W$. We show that $u \in Q$. Since $Q$ is spectral and $\mu(\mu(u_k))=\mu(u_k)$, we have $\mu(u_k) \in [Q]\subseteq Q$ and $\mu(u_k) \to \mu(u)$ by the continuity of $\mu$. Now, for each $\mu(u_k)$, there exists $x_k \in E$ such that $\lambda(x_k) = \mu(u_k)$ by Proposition \ref{prop: description of spectrality}$(a)$. Since $(u_k)$ is convergent, it is bounded. As $\lambda$ and $\mu$ are norm-preserving, for all $k$ we have
    \[ \norm{x_k} = \norm{\lambda(x_k)} = \norm{\mu(u_k)} = \norm{u_k}. \]
    So, $(x_k)$ is bounded in $\V$. As $\V$ is finite dimensional, $(x_k)$ will have a convergent subsequence. Without loss of generality, let $x_{k} \to x$ for some $x \in \V$. As $E$ is closed, we have $x\in E$.  Since $\lambda(x_{k}) \to \lambda(x)$ and $\lambda(x_{k}) = \mu(u_{k}) \to \mu(u)$, we must have $\lambda(x) = \mu(u)$, that is, $\mu(u) = \lambda(x) = \mu \big( \lambda(x) \big)$. This shows that $u \in [\lambda(E)]$. However, since $Q$ is spectral, we have $[\lambda(E)] = [\mu(Q)] = Q$ by Proposition \ref{prop: description of spectrality}$(a,c)$; hence $u \in Q$ as we wanted.\\
    
    When $E$ is open, we work with the closed set $E^{c} = \lambda^{-1}(Q^{c})$. Since $Q^c$ is spectral in $\W$, from the above, $Q^c$ is closed, i.e., $Q$ is open. 

    \smallgap
    
    $(ii)$ Now suppose that $\V$ and $\W$ are finite dimensional with $E$ compact. Then, by the continuity of $\lambda$, $\lambda(E)$ is compact in $\W$. So, the set $\mu(Q)$ (which is $\lambda(E)$ from Proposition \ref{prop: description of spectrality}$(a)$)  is compact in $\W$. Now, from Proposition \ref{prop: basic prop} (applied to the finite dimensional FTvN system $(\W,\W,\mu)$), $[\mu(Q)]=\mu^{-1}(\mu(Q))$ is compact in $\W$. As $Q=[\mu(Q)]$ from Proposition \ref{prop: description of spectrality}$(c)$, we see that $Q$ is compact in $\W$. This completes the proof.
    \qed
\end{proof}

Our next result deals with convexity issues when both $E$ and $Q$ are spectral sets in their respective spaces.

\begin{theorem} \label{thm: transfer principle2}
    Suppose $(\V,\, \W,\, \lambda)$ is a FTvN system with its reduced system $(\W,\, \W,\, \mu)$. Let $E = \lambda^{-1}(Q)$, where $Q$ is a spectral set in $\W$. Then the following statements hold. 
	\begin{itemize}
		\item [$(a)$] Suppose $\W$ is finite dimensional. If $Q$ is convex, then $E$ is convex. More generally,
		\[ \conv \lambda^{-1}(Q)\subseteq \lambda^{-1}(\conv Q). \]
		\item [$(b)$] Suppose $\V$ is a Hilbert space and $\W$ is finite dimensional. Then,
		\[ \overline{\conv}\, \lambda^{-1}(Q)=\lambda^{-1} \big( \overline{\conv}\,Q \big). \]
		\item [$(c)$] Suppose both $\V$ and $\W$ are finite dimensional. If $E$ is convex, then $Q$ is convex. Thus, in this setting, every convex spectral set in $\V$ arises as the $\lambda$-inverse image of a  convex spectral set in $\W$. 
		\item[$(d)$] Suppose both $\V$ and $\W$ are finite dimensional. If $Q$ is compact, then 
		\[ \conv \lambda^{-1}(Q) = \lambda^{-1}(\conv Q). \]
		Moreover, if $Q$ is also convex, then $\lambda^{-1}(Q)$ is compact and convex. Thus,  in this setting, every compact convex spectral set in $\V$ arises as the $\lambda$-inverse image of a compact convex spectral set in $\W$.
	\end{itemize}
\end{theorem}

\begin{proof}
	$(a)$ Suppose first that $Q$ is convex (in the finite dimensional space $\W$). We show that $E=\lambda^{-1}(Q)$ is convex. Fix $x, y \in E$ and $0 \leq t \leq 1$ in $\R$. Then, as $\W$ is finite dimensional, from Proposition \ref{prop: Lidskii type theorem}$(b)$,
	\[ u := \lambda(tx + (1-t)y) \prec t \lx + (1-t)\ly =:v. \]
	Since $\lx, \ly \in Q$ and $Q$ is convex, we have $v \in Q$. As $u \prec v$, we have $u \in \conv  [v]$; thus, we can write $u$ as a convex combination of $v_k$s, where $v_k \in [v]$. Now, as $Q$ is spectral with  $v\in Q$ and $v_k \in [v]$, we have $v_k \in Q$ for each $k$. Hence, by the convexity of $Q$, we have $u \in Q$. It follows that $tx + (1-t)y \in \lambda^{-1}(Q) = E$. This proves  the convexity of $E$ in $\V$.
	
	Now consider a general spectral set $Q$ in $\W$. As $\W$ is finite dimensional, we can apply Proposition \ref{prop: spectral invariance}$(c)$ in the system $(\W, \W, \mu)$ to see that  $\conv Q$ is convex and spectral in $\W$. Hence, by what we have proved above, $\lambda^{-1}(\conv Q)$ is convex and spectral. As
	$\lambda^{-1}(Q)\subseteq \lambda^{-1}(\conv Q)$, we see that $\conv \lambda^{-1}(Q)\subseteq \lambda^{-1}(\conv Q)$.
	
	\smallgap
	
	$(b)$ Assume that $\V$ is a Hilbert space and $\W$ is finite dimensional. As $Q$ is spectral in $\W$, by Proposition \ref{prop: spectral invariance}$(b)$, $\overline{\conv}\, Q$ is closed, convex, and spectral in $\W$. By the continuity of $\lambda$ and Item $(a)$, $\lambda^{-1}\big(\, \overline{\conv}\, Q \big)$ is closed and convex as well. Since this set contains $\lambda^{-1}(Q)$, we have 
	\begin{equation} \label{eq: closed convex hull first inclusion}
		\overline{\conv}\,\lambda^{-1}(Q)\subseteq \lambda^{-1}\big(\, \overline{\conv}\, Q \big).
	\end{equation}
	We now prove the reverse inclusion. Suppose $y \in \lambda^{-1} \big( \overline{\conv}\, Q \big)$, but $y \notin \overline{\conv} \lambda^{-1}(Q)$. Since $\V$ is a Hilbert space, by the separation theorem, there exist $c \in \V$ and an $\alpha \in \R$ such that
	\[ \ip{y}{c} > \alpha \geq \ip{x}{c} \quad \text{for all} \; x \in  \overline{\conv} \lambda^{-1}(Q). \]
	In particular, the inequality holds for all $x \in \lambda^{-1}(Q)$. Now, for each $q \in Q$, we have (from Proposition \ref{prop: description of spectrality}$(a)$) that $\mu(q) \in \mu(Q) = \lambda(E)$, so there exists (by applying $(A3)$ in Definition \ref{def: ftvn system}) $x \in \V$ such that $\lx = \mu(q)$ and
	\[ \ip{x}{c} = \ip{\lx}{\lc} = \ip{\mu(q)}{\lc}. \]
	Since $\lx = \mu(q) \in \mu(Q) \subseteq Q$ (from Proposition \ref{prop: description of spectrality}$(a)$), we have $x \in \lambda^{-1}(Q)$. Thus, for all $q \in Q$, we see that
	\[ \ip{q}{\lc} \leq \ip{\mu(q)}{\mu(\lc)} = \ip{\mu(q)}{\lc} = \ip{x}{c} \leq \alpha. \]
	As $q \in Q$ is arbitrary, by linearity and continuity of the inner product, $\ip{p}{\lc} \leq \alpha$ for all $p \in \overline{\conv}\, Q$. Specializing this to $p = \ly \in \overline{\conv}\, Q$, we get
	\[\ip{\ly}{\lc} \leq \alpha. \]
	However, this contradicts the inequalities
	\[ \alpha <\ip{y}{c} \leq \ip{\ly}{\lc}. \]
	We thus have the required reverse inclusion. Hence, the inclusion in \eqref{eq: closed convex hull first inclusion} becomes an equality.
	
	\smallgap
	
	$(c)$ Now, suppose both $\V$ and $\W$ are finite dimensional and $E = \lambda^{-1}(Q)$ is convex in $\V$. We claim  that $Q$ is convex in $\W$. Let  $u, v \in Q$ and $0 \leq t \leq 1$ in $\R$. Since $\mu(u), \mu(v) \in \mu(Q) = \lambda(E)$ by Proposition \ref{prop: description of spectrality}$(a)$, we may choose $x, y \in E$ such that 
	\[ \lx = \mu(u) \quad \text{and} \quad \ly = \mu(v). \]
	Now, put $w := tv + (1-t)u \in \W$. Then, by $(C1)$ in Definition \ref{def: reduced system}, there exists $z \in \V$ such that $\lambda(z) = \mu(w)$. In the proof given below, we will show that $z$ belongs to  $E$. Then, $\mu(w) = \lambda(z) \in \lambda(E) = \mu(Q) \subseteq Q$; hence we have $w \in Q$ due to the spectrality of $Q$, proving the convexity of $Q$.
	
	First, note that $(\W,\, \W,\, \mu)$ is a reduced system of itself (see the description after Definition \ref{def: reduced system}). Hence, applying Proposition \ref{prop: Lidskii type theorem}$(b)$ with the assumption that $\W$ is finite dimensional, we have
	\[ \lambda(z) = \mu(w) = \mu \big( tv + (1-t)u \big) \prec t\mu(v) + (1-t)\mu(u). \]
	Now, applying $(A3)$ in Definition \ref{def: ftvn system} with $q = \lx$ and $c = y$, we choose $\barx \in \V$ such that
	\[ \lambda(\barx) = \lx \quad \text{and} \quad \ip{y}{\barx} = \ip{\ly}{\lambda(\barx)}. \]
	The second equation above implies that $y$ and $\barx$ commute. Hence, $ty$ and $(1-t) \barx$ also commute; consequently, by Proposition \ref{prop: basic ftvn theorem}$(e)$, $\lambda \big( ty+(1-t) \barx \big) = t\lambda(y) + (1-t)\lambda(\barx)$. Now, since $E$ is spectral, $\lambda(\barx) = \lx$ and $x \in E$ imply that $\barx \in E$; hence we have $ty + (1-t)\barx \in E$ by the convexity of $E$. It follows that
	\begin{align*}
		\lambda(z) = \mu(w) & = \mu(tv+(1-t)u) \\
		& \prec t\mu(v) + (1-t)\mu(u) \\
		& \prec t\ly + (1-t) \lambda(\barx) \\
		& = \lambda \big( ty+(1-t) \barx \big).
	\end{align*}
	Since $\V$ is also assumed to be finite dimensional, this gives $z \prec ty + (1-t)\barx \in E$ by Proposition \ref{prop: Lidskii type theorem}$(c)$. Finally, as $E$ is convex and spectral, this implies that $z \in E$ as we have wanted.
	
	Lastly, consider a convex spectral set $E$ in $\V$. From Proposition \ref{prop: description of spectrality}$(b)$, there
	exists a spectral set $\widetilde{Q}$ in $\W$ such that $E = \lambda^{-1}(\widetilde{Q})$. Moreover, this $\widetilde{Q}$ must be convex by the argument we had earlier.
	
	\smallgap
	
	$(d)$ Since $\W$ is finite dimensional and $Q$ is compact (and spectral), it follows (see, for example, \cite{rudin1973functional}, Theorem 3.25) that $\conv Q$ is compact and convex. So, $\overline{\conv}\, Q = \conv Q$. Since $\V$ is finite dimensional, 
	$\lambda^{-1}(Q)$ is also compact in $\V$ by Proposition \ref{prop: basic prop}. Therefore, its closed convex hull is just $\conv \lambda^{-1}(Q)$. Thus, by $(b)$,
	\[ \conv \lambda^{-1}(Q) = \overline{\conv}\, \lambda^{-1}(Q) = \lambda^{-1} \big( \overline{\conv}\, Q \big) = \lambda^{-1}(\conv Q). \]
	
	Now suppose that $Q$, in addition to being compact and spectral, is also convex. Then, as noted previously, $\lambda^{-1}(Q)$ is compact. Since $\conv Q=Q$, we have $\conv \lambda^{-1}(Q) = \lambda^{-1}(\conv Q) = \lambda^{-1}(Q)$. Thus, $\lambda^{-1}(Q)$ is compact and convex in $\V$.
	
	Finally, suppose $E$ is a compact convex spectral set in $\V$. Based on the previous results, we write $E = \lambda^{-1}(\widetilde{Q})$, where $\widetilde{Q}$ is spectral and convex in $\W$.
	We now claim that $\widetilde{Q}$ is also compact. As $E$ is compact and $\lambda$ is continuous, $\lambda(E)$ is compact in $\W$. Since $\mu(\widetilde{Q}) = \lambda(E)$ (see Proposition \ref{prop: description of spectrality}$(a)$) and $\widetilde{Q} = \mu^{-1}\big( \mu(\widetilde{Q}) \big)$, we see from Proposition \ref{prop: basic prop} that $\widetilde{Q}$ is compact in $\W$. This completes the proof. \qed
\end{proof}

\begin{corollary} \label{cor: closedness}
    Suppose $\V$ and $\W$ are finite dimensional. Suppose $Q$ is a spectral convex cone in $\W$. Then, $\lambda^{-1}(Q)$ is a spectral convex cone in $\V$.
\end{corollary}

We now investigate the spectrality of the set of extreme points of a compact convex spectral set. We recall the {\it Krein-Milman theorem} (\cite{conway2019course}, Theorems 7.4 and 7.8) stated in our setting of an inner product space:
{\it Every nonempty compact convex set $K$ is the closed convex hull of its extreme points. Moreover, if $S$ is any subset of  $K$ whose closed convex hull is $K$, then $\ext(K)\subseteq \overline{S}$.}
We note that $\ext(K)$ may not be closed even in a finite dimensional space, see \cite{conway2019course}, page 148.

\begin{lemma}
Consider a FTvN system $(\V,\W,\lambda)$, where  $\V$ is finite dimensional. If $E$ is a convex spectral set in $\V$, then $\ext(E)$ and $\overline{\ext(E)}$ are  spectral in $\V$.
\end{lemma}

\begin{proof}
Assuming that $\ext(E)$ is nonempty, let $x\in \ext(E)$ and take $y \in \V$ such that $\ly = \lx$. We will show that $y \in \ext(E)$. Now, as $E$ is spectral, we have $y \in E$. If possible, let $y = \frac{a+b}{2}$, where $a,b\in E$. Using Proposition \ref{prop: convex hull of a sum},
we have 
\[ x\in [y] = \Big[ \frac{a+b}{2} \Big] \subseteq \frac{\conv [a]+\conv [b]}{2} \subseteq E, \]
where the last inclusion follows from the convexity and spectrality of $E$. Since $x\in \ext(E)$, we must have $x\in [a]$ and $x\in [b]$; so, $\ly = \lx = \lambda(a) = \lambda(b)$. This implies that 
$\norm{y} = \norm{a} = \norm{b}$. From the strict convexity of the norm, we must have $y=a=b$. This proves that $\ext(E)$ is spectral. 
Consequently, from Proposition \ref{prop: spectral invariance},
$\overline{\ext(E)}$ is also spectral. \qed
\end{proof}

\begin{theorem}
	Suppose $(\V, \W, \lambda)$ is a FTvN system with its reduced system $(\W,  \W, \mu)$, where $\V$ and $\W$ are finite dimensional.  Let $E = \lambda^{-1}(Q)$, where $Q$ is a compact convex spectral set in $\W$. Then,
	\begin{equation} \label{eq: extreme point inclusion} 
		\lambda^{-1} \big( \ext(Q) \big) \subseteq \ext \big( \lambda^{-1}(Q) \big) \subseteq \lambda^{-1} \big( \, \overline{\ext(Q)} \, \big),
	\end{equation}
	and
	\begin{equation} \label{eq: extreme point closures}
		\overline{\lambda^{-1} \big( \ext(Q) \big)} = \overline{\ext \big( \lambda^{-1}(Q) \big)} = \lambda^{-1} \big( \, \overline{\ext(Q)} \, \big).
	\end{equation}
\end{theorem}

\begin{proof}
    We prove the first inclusion in \eqref{eq: extreme point inclusion}. As $\V$ and $\W$ are finite dimensional, from Theorem \ref{thm: transfer principle2}$(b)$, we see that $E$ is compact, convex, and spectral. Let $u \in \lambda^{-1}\big (\ext(Q)\big )$ so $\lu \in \ext(Q)$. To see that $u \in \ext \big( \lambda^{-1}(Q) \big)$, that is, $u$ is an extreme point of $E$, suppose $u = \frac{a+b}{2}$, where $a, b \in E$. We will show that $u = a = b$.
	Now, as $\W$ is finite dimensional, by Proposition \ref{prop: Lidskii type theorem},
	\[ \lu =\frac{\lambda(a+b)}{2} \prec \frac{\lambda(a) + \lambda(b)}{2}. \]
	This means that 
	\[ \lu \in \frac{\conv \big [\lambda(a) + \lambda(b)\big ]}{2} \subseteq \frac{\conv [\lambda(a)] + \conv  [\lambda(b)]}{2}, \]
	where we have used Proposition \ref{prop: convex hull of a sum} (specialized to $\W$). As $[\lambda(a)]$ and $[\lambda(b)]$ are subsets of the compact convex set $Q$ (in $\W$) and $\lu \in \ext(Q)$, we see that $\lu = w_1 = w_2$, where $w_1 \in [\lambda(a)]$ and $w_2 \in [\lambda(b)]$, i.e., $\mu(w_1) = \mu(\lambda(a))$ and $\mu(w_2) = \mu(\lambda(b))$. Since $\mu \circ \lambda = \lambda$ and $\lu = w_1$, we get 
	\[ \lu = \mu(\lu) = \mu(w_1) = \mu(\lambda(a)) = \lambda(a). \]
	Similarly, we get $\lu = \lambda(b)$. Hence, $\lu = \lambda(a) = \lambda(b)$ and so $\norm{u} = \norm{a} = \norm{b}$. Now, using the strict convexity of the (inner product) norm and $u = \frac{a+b}{2}$ we see that $u = a = b$. Thus, $u \in \ext(E)$, proving the first inclusion in \eqref{eq: extreme point inclusion}. By the above Lemma, $\overline{\ext(Q)}$ is spectral in $\W$. 
	The set $\overline{\ext(Q)}$, being a closed subset of the compact convex set $Q$, is compact. Since $\W$ is finite dimensional, its convex hull, $\conv \overline{\ext(Q)}$ is compact and convex and also a subset of $Q$. In view of the Krein-Milman theorem, we have
	\[ Q=\overline{\conv} \ext(Q)\subseteq \conv \overline{\ext(Q)}\subseteq Q. \]
	Thus, we have $Q = \conv \overline{\ext(Q)}$.
	
	We now prove the second inclusion in \eqref{eq: extreme point inclusion}. Define the closed set $S$ by $S := \lambda^{-1} \big( \, \overline{\ext(Q)} \, \big)$. As $\overline{\ext(Q)}$ is compact and spectral in $\W,$ by Theorem \ref{thm: transfer principle2}$(d)$, 
	\[ \conv S = \lambda^{-1} \big( \conv \overline{\ext(Q)} \, \big) = \lambda^{-1}(Q). \]
	Consequently, the closed convex hull of $S$ is $\lambda^{-1}(Q)$; by \cite{conway2019course}, Theorem 7.8, 
	\[ \ext \big( \lambda^{-1}(Q) \big) \subseteq \overline{S} = S. \]
	Thus, we have both the inclusions in \eqref{eq: extreme point inclusion}. We now apply Theorem \ref{thm: transfer principle1}$(a)$ to the spectral set $\ext(Q)$ to see
	\[ \overline{\lambda^{-1} \big( \ext(Q) \big)} = \lambda^{-1} \big( \, \overline{\ext(Q)} \, \big). \]
	Using this and taking closures of sets in \eqref{eq: extreme point inclusion}, we get \eqref{eq: extreme point closures}. \qed
\end{proof}

\noindent{\bf Note:} It is unclear if/when the equality holds in the first inclusion in \eqref{eq: extreme point inclusion}.
\gap

The following theorem extends a result proved in the setting of Euclidean Jordan algebras, see \cite{jeong2016spectral}, Theorem 3.3.

\begin{theorem}
    Suppose $(\V, \W, \lambda)$ is a FTvN system with its reduced system $(\W,  \W, \mu)$, where $\V$ and $\W$ are finite dimensional. If $Q_1$ and $Q_2$ are convex spectral sets in $\W$, then
    \[ \lambda^{-1}(Q_1+Q_2)=\lambda^{-1}(Q_1)+\lambda^{-1}(Q_2). \]
\end{theorem}

\begin{proof}
By relying only on the finite dimensionality of $\W$, we first prove the inclusion
\begin{equation} \label{eq: first inclusion}
    \lambda^{-1}(Q_1) + \lambda^{-1}(Q_2) \subseteq \lambda^{-1}(Q_1 + Q_2).
\end{equation}
Let $u_i \in \lambda^{-1}(Q_i)$ for $i = 1, 2$. Then, we have $\lambda(u_i) \in Q_i$. Since $\W$ is finite dimensional, by Proposition \ref{prop: Lidskii type theorem}$(b)$,
\[ \lambda(u_1 + u_2) \prec \lambda(u_1) + \lambda(u_2). \]
So,
\[ \lambda(u_1 + u_2) \in \conv  [\lambda(u_1) + \lambda(u_2)] \subseteq \conv [\lambda(u_1)] + \conv  [\lambda(u_2)], \]
where the inclusion on the right comes from Proposition \ref{prop: convex hull of a sum} (applied in $\W$). As $Q_1$ and $Q_2$ are spectral and convex, we see that $\conv [\lambda(u_i)] \subseteq Q_i$ and so
\[ \lambda(u_1 + u_2) \in Q_1 + Q_2. \]
This proves \eqref{eq: first inclusion}.

\smallgap

We now prove the reverse inclusion
\begin{equation} \label{eq: second inclusion}
    \lambda^{-1}(Q_1 + Q_2) \subseteq \lambda^{-1}(Q_1) + \lambda^{-1}(Q_2).
\end{equation}
Let $E_i := \lambda^{-1}(Q_i)$ for $i = 1, 2$ and take $u \in \lambda^{-1}(Q_1 + Q_2)$ so that $\lu = q_1 + q_2$ for some $q_1 \in Q_1$ and $q_2\in Q_2$. Then, by Proposition \ref{prop: description of spectrality}$(a)$,
\[ w_i := \mu(q_i) \in \mu(Q_i) = \lambda(E_i) \subseteq Q_i, \]
for $i = 1, 2$. Hence, $w_i = \mu(q_i) = \lambda(u_i)$, where $u_i \in E_i$ for $i=1, 2$. Now, the orbit $[w_i]$, in addition to being spectral, is compact in $\W$ (as $\W$ is finite dimensional); hence, by Proposition \ref{prop: spectral invariance}$(c)$, $\conv [w_i]$ is compact, convex, and spectral in $\W$. Since $\V$ is finite dimensional, by Theorem \ref{thm: transfer principle2}, 
\[ F_i:=\lambda^{-1}\big( \conv [w_i] \big) \]
 is compact, convex, and spectral.
 Thus, $F_1+F_2$ is compact and convex. 

\smallgap

We now claim that $u \in F_1 + F_2$. If we assume the contrary, then there exist $c\in \V$ and $\alpha\in \R$ such that
\[ \ip{c}{u} > \alpha \geq \ip{c}{x_1} + \ip{c}{x_2} \;\; \text{for all} \;\; x_1 \in F_1, x_2 \in F_2. \]
Noting that $u_i\in F_i$ for $i=1,2$, we vary $x_1$ over $[u_1]$ (which is a subset of $F_1$) and $x_2$ over $[u_2]$ (a subset of $F_2$). Applying \eqref{eq: intro ftvn} and using the equality $\mu \circ \lambda = \lambda$, we see that
\begin{align*}
    \ip{c}{u} > \alpha & \geq \ip{\lambda(c)}{\lambda(u_1)} + \ip{\lambda(c)}{\lambda(u_2)} \\
    & = \ip{\lambda(c)}{\mu(q_1)} + \ip{\lambda(c)}{\mu(q_2)} \\
    & = \ip{\mu(\lambda(c))}{\mu(q_1)} + \ip{\mu(\lambda(c))}{\mu(q_2)} \\
    & \geq \ip{\lambda(c)}{q_1} + \ip{\lambda(c)}{q_2} \\
    & = \ip{\lambda(c)}{q_1 + q_2} \\
    & = \ip{\lambda(c)}{\lu} \\
    & \geq \ip{c}{u}.
\end{align*}
 This contradiction justifies the claim that $u\in F_1+F_2$. Consequently,
\[ u \in \lambda^{-1}\big( \conv [w_1] \big) + \lambda^{-1}\big( \conv [w_2] \big) \subseteq \lambda^{-1}(Q_1) + \lambda^{-1}(Q_2). \]

Since $u$ is arbitrary in $\lambda^{-1}(Q_1+Q_2)$, we get the inclusion \eqref{eq: second inclusion}. Combining this with the inclusion \eqref{eq: first inclusion}, we have the stated equality in the theorem. \qed
\end{proof}

In all of the previous results, we considered transfer principles for spectral sets. Now, we present a transfer principle for spectral functions.

\smallgap

A celebrated result of Davis \cite{davis1957all} says that a unitarily invariant function on $\Hn$ (the space of all $n\times n$ complex Hermitian matrices) is convex if and only if its restriction to diagonal matrices is convex. This result has numerous 
applications in various fields. A generalization of this result for Euclidean Jordan algebras has already been observed in \cite{baes2007convexity}. In what follows, we consider a generalization to the setting of FTvN systems. In the result below, we relate  the convexity (lower semi-continuity) of a spectral function $\Phi$ ($= \phi \circ \lambda)$ to that of $\phi$.  Recall that a real-valued function $f$ on an inner product space $X$ is lower semi-continuous if for each real number $t$, the set $\{x\in X:  f(x)\leq t\}$ is closed in $X$.

\begin{theorem}
    Suppose $(\W, \W, \mu)$ is a reduced system of $(\V, \W, \lambda)$ with a spectral function $\phi:\W\rightarrow \R$.  Let  $\Phi = \phi \circ \lambda$.  If $\W$ is finite dimensional and $\phi$ is convex (lower semi-continuous), then $\Phi$ is convex (respectively, lower semi-continuous). Moreover, if $\V$ and $\W$ are both finite dimensional, then the convexity (lower semi-continuity) of $\Phi$ implies that of $\phi$.
\end{theorem}

\begin{proof}
    Recall that $(\R \times \W, \R \times \W, M)$ with $M(t, u) = \big( t, \mu(u) \big)$ is a reduced system of $(\R \times \V, \R \times \W, \Lambda)$ with $\Lambda \big( t, x \big) = \big( t, \lx \big)$ (see Example \ref{eg: product FTvN system} for details).
    We prove our assertions by considering the epigraphs of $\Phi$ and $\phi$, denoted by $\epi \Phi$ and $\epi \phi$, respectively. Note, from Proposition \ref{prop: epigraph}, that two epigraphs are spectral sets in their respective spaces and they are related by the equality
	\begin{equation} \label{eq: epigraph}
		 \epi \Phi = \Lambda^{-1}(\epi \phi).
	\end{equation}
    Suppose first that $\phi$ is convex (resp. lower semi-continuous). Then $\epi \phi$ is a convex (resp. closed) spectral set in the FTvN system $(\R \times \W, \R \times \W, M)$. Thus, assuming $\W$ is finite dimensional, we see that $\epi \Phi$ is convex (resp. closed) by Theorem \ref{thm: transfer principle2}$(a)$. This shows that $\Phi$ is convex (resp. lower semi-continuous) on $\V$.
    
    \smallgap
    
    Conversely, if $\Phi$ is convex (resp. lower semi-continuous), then $\epi \Phi$ is a convex (resp. closed) spectral set in the FTvN system $(\R \times \V, \R \times \W, \Lambda)$. In this case, assuming both $\V$ and $\W$ are finite dimensional, we see that $\epi \phi$ is also convex (resp. closed) and spectral on $\R \times \W$ by Theorem \ref{thm: transfer principle2}$(c)$ (resp. by Proposition \ref{closedness of Q}) applied to the relation \eqref{eq: epigraph}. This proves that $\phi$ is convex (resp. lower semi-continuous) on $\W$. \qed
\end{proof}

\section{Fenchel conjugate and  subdifferential of $\phi \circ \lambda$}
Consider a FTvN system $(\V,\W,\lambda)$. In the previous sections, we only considered spectral functions that were real-valued. Now, in this section, we allow functions to take the value positive infinity. Given a function $\phi : \W \to \R \cup \{\infty\}$, we let $\Phi := \phi \circ \lambda : \V \to \R \cup \{\infty\}$ and (still) call it a `spectral function' on $\V$ (equivalently, $\Phi : \V \to \R \cup \{\infty\}$ is constant on $\lambda$-orbits.) The goal of this section is to describe the Fenchel conjugate and the subdifferential of $\phi \circ \lambda$ in terms of those of $\phi$. Motivation for these come from  similar results proved in the setting of normal decomposition systems \cite{lewis1996convex} and hyperbolic polynomials \cite{bauschke2001hyperbolic}. 

\smallgap

Let $X$ be a real inner product space. Consider 
$f:X\to \R\cup \{\infty\}$ with 
$\operatorname{dom}(f)=\{x\in X: f(x)\in \R\}$. Given $S\subseteq X$ and $\barx \in S \cap \operatorname{dom} (f) $, we define the {\it  subdifferential} of $f$ at $\barx$ relative to $S$ by
\[ \partial_{S} f(\barx) = \{ d \in X : f(x) - f(\barx) \geq \ip{d}{x - \barx} \,\, \forall x \in S \}. \]
Next, we recall the definition of {\it  Fenchel conjugate} of $f$ relative to $S$:
\[ f_{S}^{\ast}(z) = \sup_{x \in S} \big \{ \ip{x}{z} - f(x) \big \} \quad (z\in X). \]
When $S = X$, we suppress $S$ from the above notation/definitions. We recall the following proposition.

\begin{proposition} \label{prop: subgradient}
    Suppose $f:X\to \R\cup\{\infty\}$ and $S \subseteq X$ with $S \cap \operatorname{dom} (f)$ nonempty. Let
    $\barx \in S \cap \operatorname{dom} (f)$. Then the following hold:
        \begin{itemize}
            \item[$(a)$] (Fenchel inequality) $f(x) + f_{S}^{\ast}(y) \geq \ip{x}{y}$ for all $x \in S$ and $y \in X$.
            
            \item[$(b)$] $d \in \partial_{S} f(\barx)$ if and only if $f(\barx) + f_{S}^{\ast}(d) = \ip{\barx}{d}$.
        \end{itemize}
\end{proposition}

\begin{proof}
    $(a)$ This is clear from the definition of Fenchel conjugate.
    
    \smallgap 
    
    $(b)$ Suppose $d \in \partial_{S} f(\barx)$. Then from the definition, we have $f(\barx) + \ip{d}{x} - f(x) \leq \ip{d}{\barx}$ for all $x \in S$. Thus, $f(\barx) + f_{S}^{\ast}(d) \leq \ip{d}{\barx}$. Since $f(\barx) + f_{S}^{\ast}(d) \geq \ip{d}{\barx}$ holds from $(a)$, we have the desired equality.
    
    Conversely, assume $f(\barx) + f_{S}^{\ast}(d) = \ip{d}{\barx}$. Then, by $(a)$ we see that
    \[ \ip{d}{\barx} = f(\barx) + f_{S}^{\ast}(d) \geq f(\barx) + \ip{d}{x} - f(x), \]
    for all $x \in S$. Rewriting the inequality above, we get $f(x) - f(\barx) \geq \ip{d}{x - \barx}$ for all $x \in S$; hence $d \in \partial_{S} f(\barx)$. \qed
\end{proof}

\begin{theorem} \label{thm: Fenchel conjugate}
	Consider a FTvN system $(\V, \W, \lambda)$, $\phi : \W \to \R \cup \{\infty\}$, and $S \subseteq \V$. If $S$ is a spectral set in $\V$, then 
	\begin{equation}\label{eq: Fenchel conjugate 1}
		(\phi \circ \lambda)_{S}^{\ast}(z) = \phi_{\lambda(S)}^{\ast} \big( \lambda(z) \big) \quad (\forall z \in \V).
	\end{equation}
	Moreover, if  $(\W, \W, \mu)$ is a reduced system of $(\V, \W, \lambda)$ and $\phi$ is a spectral function, then
	\begin{equation}\label{eq: Fenchel conjugate 2}
		\phi_{\lambda(S)}^{\ast}(\lambda(z)) = \phi_{[\lambda(S)]}^{\ast} \big( \lambda(z) \big) \quad (\forall z \in \V).
	\end{equation}
\end{theorem}

\begin{proof} Fix $z \in \V$ and suppose $S$ is a spectral set. Since $S = [S]$, we see that
    \begin{align*}
        \phi_{\lambda(S)}^{\ast} \big( \lambda(z) \big)
        &= \sup_{w \in \lambda(S)} \Big\{\! \ip{w}{\lambda(z)} - \phi(w) \Big\} \\
        &= \sup_{x \in S} \Big\{\! \ip{\lx}{\lambda(z)} - \phi\big( \lx \big) \Big\} \\
        &= \sup_{x \in S} \Big\{\! \max_{y \in [x]} \big\{\! \ip{y}{z} - \phi\big( \lx \big) \big\} \Big\} \\
        &= \sup_{x \in S} \Big\{\! \max_{y \in [x]} \big\{\! \ip{y}{z} - \phi \big( \ly \big) \big\} \Big\} \\
        &= \sup_{y \in [S]} \Big\{\! \ip{y}{z} - \phi \big( \ly \big) \Big\} \\
        &= \sup_{y \in S} \Big\{\! \ip{y}{z} - \phi \big( \ly \big) \Big\} \\
        &= (\phi \circ \lambda)_{S}^{\ast}(z).
    \end{align*}
    (In the above, we have used the equalities $\max_{y \in [x]} \ip{y}{z} = \ip{\lx}{\lambda(z)}$ and $\phi(\ly) = \phi(\lx)$ for all $y \in [x]$.) Thus we have (\ref{eq: Fenchel conjugate 1}).
    
    \smallgap
    
    Now, let $(\W, \W, \mu)$ be a reduced system of $(\V, \W, \lambda)$. Since $\lambda(S) \subseteq [\lambda(S)]$, the inequality 
    \[ \phi_{\lambda(S)}^{\ast}\big( \lambda(z) \big) \leq \phi_{[\lambda(S)]}^{\ast} \big( \lambda(z) \big) \]
 	follows immediately. We now prove the reverse inequality when $\phi$ is spectral. By the spectrality of $\phi$, we see that $\phi(u) = \phi(\mu(u))$ for all $u \in \W$. Also, we have $\lambda(z) = \mu \big( \lambda(z) \big)$ for all $z \in \V$. It follows that $\mu([\lambda(S)]) = \lambda(S)$ and 
    \begin{align*}
        \phi_{[\lambda(S)]}^{\ast}(\lambda(z)) 
        &= \sup_{u \in [\lambda(S)]} \Big\{\! \ip{u}{\lambda(z)} - \phi(u) \Big\} \\
        &\leq \sup_{u \in [\lambda(S)]} \Big\{\! \ip{\mu(u)}{\mu(\lambda(z))} - \phi \big( \mu(u) \big) \Big\} \\
        &= \sup_{w \in \mu([\lambda(S)])} \Big\{\! \ip{w}{\lambda(z)} - \phi(w) \Big\} \\
        &= \sup_{w \in \lambda(S)} \Big\{\! \ip{w}{\lambda(z)} - \phi(w) \Big\} \\
        &= \phi_{\lambda(S)}^{\ast} \big( \lambda(z) \big).
    \end{align*}
	Thus, we have the equality \eqref{eq: Fenchel conjugate 2}, completing the proof. \qed
\end{proof}

\begin{remark}
	In the the proof of the above theorem, in particular, in that of \eqref{eq: Fenchel conjugate 1}, we relied on \eqref{eq: intro ftvn}, namely, the defining condition  of a FTvN system. It turns out by specializing \eqref{eq: Fenchel conjugate 1}, we can recover condition \eqref{eq: intro ftvn}. To see this, consider two real inner product spaces $\V$ and $\W$ with a map $\lambda : \V \to \W$. Suppose 
	\[ (\phi \circ \lambda)_{S}^{\ast}(z) = \phi_{\lambda(S)}^{\ast}\big( \lambda(z) \big), \]
	for any $\phi : \W \to \R$ and any set $S$ in $\V$ of the form $\lambda^{-1}(Q)$ with $Q\subseteq \W$.
	We specialize this condition by taking (the constant function) $\phi=0$, $z=c$, and letting for any $u\in \V$, $S=\lambda^{-1}(\{\lu\})$ (which, in our notation is $[u]$). Then,
	\[ \sup_{x\in [u]} \ip{c}{x} = (\phi \circ \lambda)_{S}^{\ast}(c) =
	\phi_{\{\lu\}}^{\ast} \big( \lc \big) = \ip{\lc}{\lu}. \]

 \gap

 The following result generalizes Theorem 4.4 in \cite{lewis1996convex} proved in the setting of normal decomposition systems.
\end{remark}

\begin{corollary} \label{cor: Fenchel conjugate 2}
    Let $(\V, \W, \lambda)$ be a FTvN system with its corresponding reduced system $(\W, \W, \mu)$. If $\phi : \W \to \R \cup \{\infty\}$ is spectral on $\W$, then 
 	\[ (\phi \circ \lambda)^{\ast} = \phi^{\ast} \circ \lambda. \] 
\end{corollary}

\begin{proof}
We let $S=\V$ in the above theorem. As the ranges of $\lambda$ and $\mu$ are the same and $\lambda = \mu \circ \lambda$, we see that for any $w\in \W$, there is an $x\in \V$ such that 
$\mu(w) = \lx = \mu \big( \lx \big)$. This means that (the above) $w$ and $\lx$ are in the same $\mu$-orbit. It follows that $\W=[\lambda(\V)]$. We now have
\[ (\phi \circ \lambda)_{\V}^{\ast}(z) = \phi_{\lambda(\V)}^{\ast} \big( \lambda(z) \big) = \phi_{[\lambda(\V)]}^{\ast} \big( \lambda(z) \big) = \phi^*_{\W} \big( \lambda(z) \big). \]
By suppressing $\V$ and $\W$, we get $(\phi \circ \lambda)^{\ast} = \phi^{\ast} \circ \lambda$. \qed
\end{proof}

Our next result deals with a subdifferential formula. We recall the definition of commutativity: $x$ and $y$ commute in $\V$ if $\ip{x}{y} = \ip{\lx}{\ly}$, or equivalently, $\lambda(x+y) = \lx + \ly$.

\begin{theorem} \label{thm: subdifferential formula}
	Let $(\V, \W, \lambda)$ be a FTvN system. Suppose $\Phi=\phi\circ \lambda$, where $\phi : \W \to \R \cup \{\infty\}$. Let $S \subseteq \V$ be a spectral set and $\barx \in S \cap \operatorname{dom}(\Phi)$. Then,
	\begin{equation}\label{eq: first subdifferential formula} y \in \partial_{S} \Phi(\barx) \Longleftrightarrow \ly \in \partial_{\lambda(S)} \phi \big (\lambda(\barx)\big ) \text{ and $y$ commutes with $\barx$}. 
 \end{equation}
Additionally, if $(\W,\W,\mu)$ is a reduced system of $(\V,
\W,\lambda)$ and $\phi$ is spectral, then,
	\[ y \in \partial_{S} \Phi(\barx) \Longleftrightarrow \ly \in \partial_{[\lambda(S)]} \phi \big (\lambda(\barx)\big ) \text{ and $y$ commutes with $\barx$}. \]
 In particular, $y \in \partial \Phi(\barx) \Longleftrightarrow \ly \in \partial \phi \big (\lambda(\barx) \big)$ and $y$ commutes with $\barx$. 
\end{theorem}

\begin{proof}
    First, assume $\ly \in \partial_{\lambda(S)} \phi (\lambda(\barx))$ and the commutativity of $\barx$ and $y$. Then, we have
    \[ \ip{\barx}{y} = \ip{\lambda(\barx)}{\ly} = \phi\big (\lambda(\barx)\big ) + \phi_{\lambda(S)}^{\ast}(\ly), \]
    by Proposition \ref{prop: subgradient}$(b)$ applied to $\phi$. Since $\phi_{\lambda(S)}^{\ast}(\ly) = \Phi_{S}^{\ast}(y)$ from Theorem \ref{thm: Fenchel conjugate}, we see that
    \[ \ip{\barx}{y} = \Phi(\barx) + \Phi_{S}^{\ast}(y), \]
    implying $y \in \partial_{S} \Phi(\barx)$.

    Conversely, given $y \in \partial_{S} \Phi(\barx)$, we have $\Phi(\barx) + \Phi_{S}^{\ast}(y) = \ip{\barx}{y}$ by Proposition \ref{prop: subgradient}$(b)$. Since $S$ is assumed to be spectral, we have $\Phi_{S}^{\ast}(y) = \phi_{\lambda(S)}^{\ast}(\ly)$;  so,
    \begin{equation} \label{eq: subgradient condition}
        \phi \big( \lambda(\barx) \big) + \phi_{\lambda(S)}^{\ast} \big( \ly \big) = \ip{\barx}{y} \leq \ip{\lambda(\barx)}{\ly}.
    \end{equation}
    On the other hand, note that $\phi(\lambda(\barx)) + \phi_{\lambda(S)}^{\ast}(\ly) \geq \ip{\lambda(\barx)}{\ly}$ by Proposition \ref{prop: subgradient}$(a)$.  Hence, we have the equality throughout  \eqref{eq: subgradient condition}, i.e.,
    \[ \phi \big( \lambda(\barx) \big) + \phi_{\lambda(S)}^{\ast} \big( \ly \big) = \ip{\barx}{y} = \ip{\lambda(\barx)}{\ly}. \]
    The equalities above imply that $\ly \in \partial_{\lambda(S)} \phi \big( \lambda(\barx) \big)$ and the commutativity of $\barx$ and $y$. Thus we have proved the equivalence in (\ref{eq: first subdifferential formula}). 

    \smallgap

    For the second part, we (further) assume that  $\phi$ is spectral on a reduced system $(\W, \W, \mu)$. Since $\Phi_{S}^{\ast}(y) = \phi_{[\lambda(S)]}^{\ast}(\ly)$ by Theorem \ref{thm: Fenchel conjugate}, the same argument above with $\lambda(S)$ replaced by $[\lambda(S)]$ yields the desired conclusion.

    \smallgap
The final statement comes from the second part by taking  $S = \V$ and noting $\W = [\lambda(\V)]$. \qed
\end{proof}

\begin{remark}
    As noted in the Introduction, in a FTvN system, one has \eqref{eq: linear sup}, which, when specialized, yields (\ref{eq: intro ftvn}). Changing $\phi$ to $-\phi$ in \eqref{eq: linear sup} we get the equality
    \begin{equation} \label{eq: modified linear sup} 
        \sup_{x \in S} \Big\{\! \ip{y}{x} - \Phi(x) \Big\} = \sup_{u \in \lambda(S)} \Big\{\! \ip{\ly}{u} - \phi(u) \Big\}, 
    \end{equation}
    where $S \subseteq \V$ is a spectral set, $\Phi = \phi\circ \lambda : \V \to \R$ is a spectral function, and $y \in \V$. Moreover, if $\barx$ solves the problem on the left, then $\barx$ commutes with $y$ and the maximum value is given by $\ip{\ly}{\lambda(\barx)} - \phi \big( \lambda(\barx) \big)$.  We now remark that when $\phi$ is real-valued, \eqref{eq: modified linear sup}  is equivalent to \eqref{eq: first subdifferential formula}. To see this, given $\barx,\, y \in \V$, notice that $y \in \partial_{S} \Phi(\barx)$ if and only if $\barx$ is a maximizer of the problem $\sup_{x \in S} \big\{\! \ip{y}{x} - \Phi(x) \big\}$. Similarly, we have $\lambda(\barx) \in \partial_{\lambda(S)} \phi \big( \lambda(\barx) \big)$ if and only if $\lambda(\barx)$ is a maximizer of the problem $\sup_{u \in \lambda(S)} \big\{\! \ip{\ly}{u} - \phi(u) \big\}$. These justify that \eqref{eq: modified linear sup} is equivalent to \eqref{eq: first subdifferential formula}. Finally, in view of an earlier remark, we conclude that \eqref{eq: intro ftvn}, \eqref{eq: linear sup}, \eqref{eq: Fenchel conjugate 1}, and \eqref{eq: first subdifferential formula} are all equivalent. 
\end{remark}

We note a simple consequence of the above theorem.

\begin{corollary}[Geometric commutation principle, \cite{gowda2017commutation}] 
    Suppose $S$ is a spectral set in the FTvN system $(\V,\W,\lambda)$ and $\overline{x}\in S$. Then, every element in the normal cone of $S$ at $\overline{x}$ commutes with $\overline{x}$.
\end{corollary}

\begin{proof}
Recall that the normal cone of $S$ at $\overline{x}$ is the set
\[ \big\{ d \in \V: \ip{d}{x - \barx} \leq 0 \,\,\forall x \in S \big\}. \]
Every $d$ in this set belongs to $\partial_{S}(\phi\circ \lambda)(\overline{x})$, where $\phi$ is the constant function zero. By the above theorem, $d$ commutes with $\overline{x}$. \qed
\end{proof}

\begin{remark}
    Throughout this paper, we have considered spectral functions defined on the entire space. In some applications, we may need to consider functions that are partially defined (say, on a subset); for instance, in the proof of Theorem \ref{thm: Fenchel conjugate}, it is sufficient to assume that $\phi$ is spectral on $[\lambda(S)] \subseteq \W$. In such a situation, we can extend the partially defined function to the entire space still maintaining its spectrality. Here is a brief justification: 
    Given a spectral set $E = \lambda^{-1}(Q)$ in $(\V,\W,\lambda)$ with $Q\subseteq \W$, let  $\Phi : E \to \R$ be spectral on $E$, meaning that   there exists a function $\phi : Q \to \R$ such that $\Phi = \phi \circ \lambda$ on $E$. 
     Then, as $\lambda(E^c)\cap Q=\emptyset$, we can extend $\phi$ to all of $\W$ by defining its values on $Q^c$ arbitrarily. By composing this extension with $\lambda$, we get an  extension of $\Phi$ to all of $\V$. When the given FTvN system has a reduced system, say,  $(\W, \W, \mu)$, we can modify this construction. Starting with $E=\lambda^{-1}(Q)$, we rewrite $E$ as $E = \lambda^{-1}(\widetilde{Q})$, where $\widetilde{Q} = \big[ Q \cap \mu(Q) \big]$ is spectral in $\W$, see Proposition \ref{prop: description of spectrality}$(a)$. Observing that $\Phi(x)=\phi \big( \lx \big) = \widetilde{\smash[t]\phi} \big( \lx \big)$ for all $x\in E$, we follow the construction given above to extend $\widetilde{\smash[t]\phi}$ to all of $\W$ and correspondingly extend $\Phi$ to all of $\V$. 
     Clearly, appropriate modifications can be done if the functions involved are extended real-valued.
\end{remark}

\section{Concluding Remarks}
In this article, working in the unified framework of Fan-Theobald-von Neumann systems, we presented transfer principles dealing with topological/convexity properties of spectral sets and functions. We also presented Fenchel conjugate and subdifferential formulas, generalizing the results of Lewis and Bauschke et al. proved in the settings of normal decomposition systems and hyperbolic polynomials.


\gap\gap

\textbf{Acknowledgements.} The work of J. Jeong was supported by the National Research Foundation of Korea NRF-2021R1C1C2008350.


\end{document}